\documentclass[11pt]{article}
\usepackage[left=1in,right=1in,top=1in,bottom=1in]{geometry}
\usepackage{setspace}
\usepackage{amsmath}
\usepackage{amsthm}
\usepackage{amssymb}
\makeatletter

\renewcommand{\@fnsymbol}[1]{\@alph{#1}}

\newcommand{\bbr}{\mathbb{R}}

\newcommand{\bbn}{\mathbb{N}}

\newcommand{\bbx}{\mathbb{X}}
\newcommand{\bby}{\mathbb{Y}}

\newcommand{\fn}[1]{\footnote{#1}}

\newcommand{\pcal}{\mathcal{P}}

\newcommand{\bcal}{\mathcal{B}}

\newcommand{\ncal}{\mathcal{N}}

\newcommand{\mcal}{\mathcal{M}}

\makeatother
\newtheorem{lemma}{Lemma}[section]
\newtheorem{proposition}[lemma]{Proposition}
\newtheorem{theorem}[lemma]{Theorem}
\newtheorem{corollary}[lemma]{Corollary}
\newtheorem{definition}[lemma]{Definition}
\newtheorem{example1}[lemma]{Example}
\newtheorem{ex1}[lemma]{Example}
\newtheorem{rem1}[lemma]{Remark}
\newtheorem{assumption}[lemma]{Assumption}
\newtheorem{alg1}[lemma]{Algorithm}
\newtheorem{me1}[lemma]{Mechanism}

\newenvironment{remark}{\begin{rem1}\rm}{\end{rem1}}

\newenvironment{example}{\begin{example1}\rm}{\end{example1}}

\numberwithin{equation}{section}
\numberwithin{figure}{section}

\usepackage{color}

\newcommand{\cl}{{\rm cl\,}}
\newcommand{\graph}{{\rm graph\,}}

\DeclareMathOperator*{\FIX}{FIX}
\DeclareMathOperator*{\argmin}{arg\,min}
\DeclareMathOperator*{\argmax}{arg\,max}

\begin{document}

\title{Continuity and sensitivity analysis of parameterized Nash games}
\author{Zachary Feinstein\fn{Stevens Institute of Technology, School of Business, Hoboken, NJ 07030, USA, {\tt zfeinste@stevens.edu}.}\\[0.7ex] \textit{Stevens Institute of Technology}.}
\date{\today}
\maketitle

\begin{abstract}
In this paper we consider continuity of the set of Nash equilibria and approximate Nash equilibria for parameterized games.  
For parameterized games with unique Nash equilibria, the continuity of this equilibrium mapping is well-known.  However, when the equilibria need not be unique, there may exist discontinuities in the equilibrium mapping.  The focus of this work is to summarize continuity properties for parameterized Nash equilibria and prove continuity via the approximate Nash game with uniformly continuous objective functions over potentially non-compact strategy spaces.
\end{abstract}\vspace{0.2cm}

\section{Introduction}\label{Sec:Intro}
Mathematical and economic games are, traditionally, defined with fixed parameters.  However, in reality these parameters are often calibrated from data, e.g., for risk aversion for the players of the game.  This naturally introduces estimation errors in the best response function and, thus, also the Nash equilibria.  In this work, we consider the convergence of the set of Nash equilibria as parameters limit to the ``true'' values.

The results determined herein are directly motivated by mean field games.  Such problems present a setting with a finite number of players, but with \emph{non-compact} strategy spaces.  As far as the author is aware, no results on essential equilibria and robustness of Nash equilibria have previously been general enough to consider such games.
The results of this work would, equally, be applicable for a differential game as the limit of discrete time dynamic games as the time step tends to 0.  In this work we are focused on the set of Nash equilibria themselves rather than the value of the game as is considered in, e.g.,~\cite{FRZ2020dpp}.

By the nature of games, this problem is intimately related to the question of continuity of fixed points.  That question has been studied in the literature before.  For single-valued mappings we refer to, e.g., \cite{browder60,gauld88,Kwiecinski92,ladeira92,MR94}, for set-valued or multivalued mappings we refer to, e.g., \cite{markin73,markin76,lim85,kirr97,RPS03,EP05,PRS14}.  Some of these results are summarized in Appendix~\ref{Sec:FixedPt}.

This work studies problems in the vein of essential equilibria, which are those Nash equilibria which are robust to perturbations.  This concept was introduced in~\cite{wu1962essential} and has been extended in, e.g.,~\cite{yu1999essential,cn2010essential}.  In this work, we similarly consider the robustness of the set of equilibria to perturbations in some parameter of the game.  Herein we extend these notions of robustness for non-compact spaces, which is vital for studying, e.g., mean field games.

The primary innovation of this work is summarized in Figure~\ref{Fig:NashCont}.  The notation and details for this summary are provided in Section~\ref{Sec:Nash}.  Briefly, consider a Nash game conditional on some parameter, e.g., the risk aversion for every player.  Let $h(\cdot)$ denote the set of parameterized, pure or mixed, Nash equilibria (in a, potentially, \emph{non-compact} metric space) and 
$h^{\vec\epsilon}(\cdot)$ denote
the set of approximate Nash equilibria for the same game.  Figure~\ref{Fig:NashCont} provides relations between these equilibria as the parameters converge.  Notably, and as highlighted throughout this work, the full set of Nash equilibria is captured by convergence of the approximate equilibria rather than the set of equilibria themselves.  This result is provided in Theorem~\ref{Thm:Approx-Limit}.  This has implications for data dependence since consideration of the Nash equilibria themselves can result in an under-accounting of the true equilibria. This is highlighted for a mean field game in~\cite{nutz2018casestudy}.
\begin{figure}[t]
\centering
$\boxed{\liminf_{n \to \infty} h(x_n) \subseteq \limsup_{n \to \infty} h(x_n) \subseteq h(x) = \lim_{\vec\epsilon \searrow 0} \limsup_{n \to \infty} h^{\vec\epsilon}(x_n) = \lim_{\vec\epsilon \searrow 0} \liminf_{n \to \infty} h^{\vec\epsilon}(x_n)}$
\caption{Ordering for the convergence of fixed points and approximate fixed points with $x_n \to x$.}
\label{Fig:NashCont}
\end{figure}
Throughout this work we consider the setting in which there are a multiplicity of equilibria.  In the case of unique parameterized equilibria over compact spaces, continuity of the Nash equilibria is a fairly standard result.

The organization of this paper is as follows.  In Section~\ref{Sec:Motivation}, we present a simple two player game that conceptually demonstrates the major results of this work.  Properly motivated, we present the main results of this work in Section~\ref{Sec:Nash}.  For completeness, continuity arguments for the set of Nash equilibria are presented in Section~\ref{Sec:Nash-Games}.  The main results of this work are then presented in Section~\ref{Sec:Nash-Approximate}; it is in this section that we demonstrate how to utilize the approximate Nash equilibria to find the appropriate continuity and convergence notions.  Background material on set-valued continuity are presented in the Appendix~\ref{Sec:Background}. Theoretical results on the continuity of the value and best response functions for a game are summarized in Appendix~\ref{Sec:BRF}.  Considerations on the continuity of fixed point problems are summarized in Appendix~\ref{Sec:FixedPt}.

\section{Motivating example}\label{Sec:Motivation}
Before considering the main results of this work in Section~\ref{Sec:Nash}, we wish to provide a simple motivating example for the continuity and convergence arguments undertaken in this work.  Specifically, to consider a simple setup, we introduce a parameterized two player game in a compact space.  
In this example we use some terminology of Nash games, for descriptions of these terms we refer the reader to Section~\ref{Sec:Nash-Setting} below.

Consider a simple two player game constructed to provide many of the fundamental insights seen in Section~\ref{Sec:Nash-Games}.  For this problem we consider a compact parameter space $\bbx := [0,2]$ over which we want to consider the continuity of our Nash equilibria.  The strategy space for both players are the compact spaces $\bby_1 = \bby_2 := [0,1]$ and both players are seeking to maximize strictly concave quadratic objective functions.  Let $H: [0,2] \times [0,1]^2 \to [0,1]^2$ denote the best response function for this game; specifically we define this game by:
\begin{align*}
H_1(x,y) &:= \argmax_{y_1^* \in [0,1]} \left\{-y_1^*(y_1^* - 2xy_2)\right\} = \min(xy_2,1)\\
H_2(x,y) &:= \argmin_{y_2^* \in [0,1]} \left\{-y_2^*(y_2^* - 2y_1)\right\} = y_1.
\end{align*}
By Brouwer's fixed point theorem, there exists some Nash equilibrium for every $x \in [0,2]$, though it need not be unique.  In fact, for any $x \in [0,2]$, we define the set of Nash equilibria (i.e., the fixed points) by the mapping:
\begin{align*}
h(x) &:= \FIX_{y \in [0,1]^2} H(x,y) = \begin{cases} \{(0,0)\} &\text{if } x < 1 \\ \{y \in [0,1]^2 \; | \; y_1 = y_2\} &\text{if } x = 1 \\ \{(0,0) \; , \; (1,1)\} &\text{if } x > 1. \end{cases}
\end{align*}
By the closed graph theorem (see, e.g., Theorem~\ref{Thm:ClosedGraphThm}) this fixed point mapping is upper continuous.  However, it is \emph{not} lower continuous at $x = 1$.  This can be demonstrated by setting $x = 1$ and $y \in (0,1)^2$ such that $y_1 = y_2$; for any $\{x_n\}_{n \in \bbn} \subseteq [0,2] \backslash \{1\} \to 1$ and any $\{y_n\}_{n \in \bbn} \subseteq [0,1]^2 \to y$ there exists $N \geq 0$ such that $y_n \not\in h(x_n)$ for every $n \geq N$.

Intuitively, this provides the main interpretations of upper and lower continuity, i.e.,
\begin{itemize}
\item no point in an upper continuous set-valued mapping ``disappears'' except as the limit of a net and
\item no point in a lower continuous set-valued mapping ``appears'' except as the limit of other values.
\end{itemize}
Notably, these continuity arguments coincide for single-valued functions.  For more details we refer to Appendix~\ref{Sec:Background} and the references cited therein.
In fact, as discussed in Appendix~\ref{Sec:FixedPt}, we often find that fixed point mappings are upper, but not lower, continuous.  Thus, as presented in Proposition~\ref{Prop:ClosedLimit}, there may be fixed points that are not limits of fixed points as seen in this simple example.

Before presenting the main theoretical results, we wish to give a quick hint as to the main results of this work.  If instead of considering the original game and associated fixed point problem $y = H(x,y)$, we introduce an approximating game for $\epsilon \in (0,\frac{1}{4})$:
\begin{align*}
H_1^{\epsilon}(x,y) &:= \left\{y_1^* \in [0,1] \; | \; -y_1^*(y_1^*-2xy_2) > v_1(x,y) - \epsilon\right\}\\
    &= \left(xy_2 - \sqrt{\max\{0,xy_2-1\}^2 + \epsilon} \; , \; xy_2 + \sqrt{\epsilon}\right) \cap [0,1]\\
H_2^{\epsilon}(x,y) &:= \left\{y_2^* \in [0,1] \; | \; -y_2^*(y_2^*-2y_1) > v_2(x,y) - \epsilon\right\}\\
    &= \left(y_1 - \sqrt{\epsilon} \; , \; y_1 + \sqrt{\epsilon}\right) \cap [0,1]
\end{align*}
where $v_1(x,y) = -\min\{xy_2,1\}^2 + 2xy_2\min\{xy_2,1\}$ and $v_2(x,y) = y_1^2$ are the optimal values for the payoff of player 1 and 2 respectively given the parameter $x$ and the strategy of the other player.
With this fixed approximation error $\epsilon$, and with the convention that $1/0 = \infty$, the set of approximating Nash equilibria $h^{\epsilon}(x) := \FIX_{y \in [0,1]^2} H^{\epsilon}(x,y)$ are provided by:
\begin{align*}
h^{\epsilon}(x) &= \begin{cases}
    A(x) &\text{if } x \leq 1 \\
    A(x) \cup B_1(x) &\text{if } x \in (1 , \frac{1+\epsilon}{(1-\sqrt{\epsilon})^2}) \\
    A(x) \cup B_2(x) \cup C(x) &\text{if } x \geq \frac{1 + \epsilon}{(1-\sqrt{\epsilon})^2},
    \end{cases}\\
A(x) &= \left\{y \in [0,1]^2 \; | \; y_1 \in (\max\{1,x\}y_2 - \sqrt{\epsilon} , \min\{1,x\}y_2 + \sqrt{\epsilon}) , y_2 \in [0,\frac{2}{|1-x|}\sqrt{\epsilon})\right\},\\
B_1(x) &= \left\{y \in [0,1]^2 \; | \; y_1 \in (xy_2-\sqrt{(xy_2-1)^2+\epsilon} , y_2 + \sqrt{\epsilon}) , y_2 \in (\frac{1}{x} , 1]\right\},\\
B_2(x) &= \left\{y \in [0,1]^2 \; \left| \; \begin{array}{l} y_1 \in (xy_2-\sqrt{(xy_2-1)^2+\epsilon} , y_2 + \sqrt{\epsilon}) ,\\ y_2 \in (\frac{1}{x} , \frac{(1-\sqrt{\epsilon})x + \sqrt{\epsilon} - \sqrt{[(1-\sqrt{\epsilon})x+\sqrt{\epsilon}]^2 - [2x-1]}}{2x-1}) \end{array}\right.\right\},\\
C(x) &= \left\{y \in [0,1]^2 \; \left| \; \begin{array}{l} y_1 \in (xy_2-\sqrt{(xy_2-1)^2+\epsilon} , y_2 + \sqrt{\epsilon}) ,\\ y_2 \in (\frac{(1-\sqrt{\epsilon})x + \sqrt{\epsilon} + \sqrt{[(1-\sqrt{\epsilon})x+\sqrt{\epsilon}]^2 - [2x-1]}}{2x-1} , 1] \end{array}\right.\right\}.
\end{align*}
We wish to note that $B_1(x) = \emptyset$ for $x = 1$, $B_1(x) = B_2(x) \cup C(x) \cup \{(\frac{1 + \sqrt{\epsilon} + \epsilon + \epsilon\sqrt{\epsilon}}{(1+\sqrt{\epsilon})^2} , \frac{1-\epsilon}{(1+\sqrt{\epsilon})^2})\}$ for $x = \frac{1+\epsilon}{(1-\sqrt{\epsilon})^2}$, and $B_2(x) = \emptyset$ for $x \geq \frac{1}{1-2\sqrt{\epsilon}}$.

By construction, $h(x) = \lim_{\epsilon \searrow 0} h^{\epsilon}(x)$ for any $x \in [0,2]$.  Additionally, and providing an application of Theorem~\ref{Thm:Approx-Limit}, $h(x) = \lim_{\epsilon \searrow 0} \liminf_{n \to \infty} h^{\epsilon}(x_n)$ for any sequence $\{x_n\}_{n \in \bbn} \subseteq [0,2] \to x$.  Thus we can consider sensitivity analysis on the approximating fixed points $h^{\epsilon}$ in order to study the set of fixed points $h$.  This is in contrast to how such objects are typically studied and motivates the consideration of approximations taken within this work more generally.

\section{Continuity of the set of Nash equilibria}\label{Sec:Nash}
This section provides all the main results of this work.  This is summarized in Figure~\ref{Fig:NashCont} for games where the set of parameterized Nash equilibria is given by $h(\cdot)$.
The notation used in the summary Figure~\ref{Fig:NashCont} is provided in Section~\ref{Sec:Nash-Setting} with the details provided in Sections~\ref{Sec:Nash-Games} and~\ref{Sec:Nash-Approximate}.

\subsection{Setting and notation}\label{Sec:Nash-Setting}
Consider a game whose set of players is indexed by $I$.  For simplicity, throughout this work we will consider a finite player game $|I| < \infty$ or one that can be modeled in such a way (e.g., a mean field game).  Let $\bby_i$, the strategy space (either of pure or mixed strategies) for player $i \in I$, be a metric space with metric $d_i$.  Define $\bby := \prod_{i \in I} \bby_i$ to be the product space with metric $d(y^1,y^2) := \sum_{i \in I} d_i(y^1_i,y^2_i)$.  Denote the power set of $\bby$ by $\pcal(\bby) := \{Y \subseteq \bby\}$ (respectively $\pcal(\bby_i)$ of $\bby_i$).
Throughout this work, we often want to fix the strategy of all players but $i$.  Notationally we define such a strategy by $y_{-i} \in \prod_{j \in I\backslash\{i\}} \bby_j =: \bby_{-i}$, and (with slight abuse of notation) $y = (y_i,y_{-i})$.
As we wish to consider a parameterized Nash game, we introduce the parameter space $\bbx$ to be a metric space.

Consider a parameterized Nash game in which player $i \in I$ has continuous utility function $f_i: \bbx \times \bby \to \bbr$ and continuous strategy space $F_i: \bbx \times \bby_{-i} \to \pcal_f(\bby_i) := \{Y_i \subseteq \bby_i \; | \; Y_i = \cl(Y_i), \; Y \neq \emptyset\}$.  
\begin{definition}\label{Defn:NashEquil}
Given a parameterized Nash game, a \textbf{\emph{Nash equilibrium}} $y^* \in \bby$ at $x \in \bbx$ if, for every $i \in I$, $y_i^* \in F_i(x,y_{-i}^*)$ and
\[f_i(x,y^*) \geq f_i(x,y_i,y_{-i}^*) \quad \forall y_i \in \bby: \; y_i \in F_i(x,y_{-i}^*).\]
\end{definition}

Oftentimes the Nash equilibria are defined in terms of the best response function rather than the definition given above.
\begin{definition}\label{Defn:BestRespFnc}
Given a parameterized Nash game, the \textbf{\emph{best response function}} for player $i \in I$ is defined as $H_i: \bbx \times \bby \to \pcal(\bby_i)$ with
\begin{equation}\label{Eq:H}
H_i(x,y) := \argmax_{y_i^* \in \bby_i} \left\{f_i(x,y_i^*,y_{-i}) \; | \; y_i^* \in F_i(x,y_{-i})\right\}.
\end{equation}
\end{definition}
Note that a best response function, $H_i$, may be empty for some choices of $(x,y) \in \bbx \times \bby$.
With this function we are able to determine the set of all Nash equilibria as a fixed point problem.
Consider the product function $H := \prod_{i \in I} H_i$ of the best response functions for our Nash game.
\begin{definition}\label{Defn:NashEquil2}
Given a parameterized Nash game and best response function $H := \prod_{i \in I} H_i$, the parameterized \textbf{\emph{set of Nash equilibria}} is the mapping $h: \bbx \to \pcal(\bby)$ such that
\begin{equation}\label{Eq:h}
h(x) := \FIX_{y \in \bby} H(x,y) := \{y \in \bby \; | \; y \in H(x,y)\}.
\end{equation}
\end{definition}
It is this mapping $h:\bbx \to \pcal(\bby)$ of Nash equilibria which we wish to investigate within this work.
Throughout the remainder of this work we will utilize the notation and definitions provided in this section.

\subsection{Equilibria of Nash games}\label{Sec:Nash-Games}
In this section we study the best response function $H$ and set of Nash equilibria $h$ to determine continuity and sensitivity arguments for the set of Nash equilibria.  In particular, as shown in the motivating example of Section~\ref{Sec:Motivation}, we find that the set of Nash equilibria are typically not continuous and therefore the results highlighted by that example hold more generally. These results are generally well-known and included for completeness.  To ease the readability of this work, we first review the assumptions from Section~\ref{Sec:Nash-Setting} which will hold throughout the rest of this work.

\begin{assumption}\label{Ass:Nash-1}
Consider a \emph{finite} player game with players indexed by $I$.  The strategy space of player $i \in I$ is given by the metric space $\bby_i$.  Additionally, the game is parameterized by the metric space $\bbx$.  Each player $i \in I$ has continuous utility function $f_i: \bbx \times \bby \to \bbr$ and continuous feasible strategy space $F_i: \bbx \times \bby_{-i} \to \pcal_f(\bby_i)$.
\end{assumption}

\begin{proposition}\label{Prop:NashClosed}
Consider the generalized Nash game described in Section~\ref{Sec:Nash-Setting}, $\graph H \subseteq \bbx \times \bby \times \bby$ and $\graph h \subseteq \bbx \times \bby$ are closed in their respective product topologies.
\end{proposition}
\begin{proof}
First, we note by Theorem~\ref{Thm:MaxGraph} that $\graph H_i$ is closed for every player $i \in I$.  Second, we determine that
\[\graph H = \{(x,y,y^*) \in \bbx \times \bby \times \bby \; | \; y_i^* \in H_i(x,y) \; \forall i \in I\} = \bigcap_{i \in I} \left[\graph H_i \times \bby_{-i}\right]\]
is closed as it is the intersection of closed sets.
Finally, by Lemma~\ref{Lemma:FixedPtUC}\eqref{Lemma:FixedPtUC-1}, $\graph h$ is closed as well.
\end{proof}

Utilizing the above result on the closedness of the graph of parameterized Nash equilibria $h$, we can consider the convergence of equilibria.
\begin{corollary}\label{Cor:NashClosed}
Let $(x_n)_{n \in \bbn} \to x$ be a convergent sequence of parameters in the domain of $h$, i.e., so that $h(x_n),h(x) \neq \emptyset$ for every $n \in \bbn$.  Then
$h(x) \supseteq \limsup_{n \to \infty} h(x_n)$. 
\end{corollary}
\begin{proof}
This follows directly from Propositions~\ref{Prop:NashClosed} and~\ref{Prop:ClosedLimit}.
\end{proof}

\begin{remark}\label{Rem:Unique}
Throughout this work we have assumed a multiplicity of Nash equilibria.  However, if $\operatorname{card}[h(\cdot)] = 1$ and $\bby$ is a compact space, then $h(x) = \lim_{n \to \infty} h(x_n)$ for any sequence $(x_n)_{n \in \bbn} \to x$.  To prove continuity of $h$, we use the closed graph theorem (see, e.g.,~\cite[Theorem 2.58]{AB07}).
\end{remark}

\begin{remark}\label{Rem:NashClosed}
While we assumed throughout this section that $\bbx$ and $\bby$ are metric spaces, these are stronger conditions than necessary for the results provided so far.  Proposition~\ref{Prop:NashClosed} only requires that $\bbx,\bby$ be Hausdorff spaces and Corollary~\ref{Cor:NashClosed} requires additionally that $\bby$ is a regular space.  As a consequence, the finite player assumption $|I| < \infty$ can also be dropped for the results of this section.
\end{remark}

\begin{example}\label{Ex:Nutz}
\cite{nutz2018casestudy} studies the mean field limit of a sequence of $n$ player games (as $n \to \infty$).  We wish to demonstrate how Corollary~\ref{Cor:NashClosed} (with generalizations as noted in Remark~\ref{Rem:NashClosed}) can immediately provide the main result of that paper while referring the reader to \cite{nutz2018casestudy} for the formalities and setup.

Consider a game with an infinite number of players indexed by the interval $I = [0,1]$.  Let $\bcal$ be the Borel $\sigma$-algebra of the player space $I$.  For this example we will parameterize our game by the measure of player importance on the equilibrium, i.e., consider $\bbx := \mcal$ to be the space of all probability measures on the measurable space $(I,\bcal)$ absolutely continuous w.r.t.\ the Lebesgue measure $\lambda$.  Leaving out the formalities of the game, we are interested in the set of Nash equilibria $h(\mu)$ for $\mu \in \mcal$; specifically, we are interested in the relation between the $n$ player game and the mean field limit, i.e., between discrete measures and the Lebesgue measure $\lambda \in \mcal$.  Consider the sequence of discrete probability measures $(\lambda_n)_{n \in \bbn} \subseteq \mcal$ defined by
\[\lambda_n(B) := \frac{1}{n} \sum_{k = 1}^n \mathbb{I}(\frac{k}{n} \in B) \quad \forall B \in \bcal.\]
Utilizing the metric $d_\mcal(\mu_1,\mu_2) := \sup_{\substack{a,b \in I \\ a \leq b}} |\mu_1([a,b]) - \mu_2([a,b])|$ on $\mcal$, $\lim_{n \to \infty} d_\mcal(\lambda_n,\lambda) = 0$ by definition.
A key result of \cite{nutz2018casestudy}, whose inclusion follows from Corollary~\ref{Cor:NashClosed}, is 
\[\left\{\lim_{n \to \infty} \rho_n \; | \; \rho_n \in h(\lambda_n) \; \forall n \in \bbn\right\} =: \liminf_{n \to \infty} h(\lambda_n) \subsetneq h(\lambda).\]
\end{example}

\subsection{Approximate Nash equilibria}\label{Sec:Nash-Approximate}
Consider now a game such that $\{(x,y_{-i}) \mapsto f_i(x,y_i^*,y_{-i})\}_{y_i^* \in \bby_i}$ is uniformly equicontinuous.
Define $f_i^\epsilon: \bbx \times \bby \to [-\frac{1}{\epsilon} , \frac{1}{\epsilon}]$ with $f_i^\epsilon(x,y) := -\frac{1}{\epsilon} \vee f_i(x,y) \wedge \frac{1}{\epsilon}$ for any $\epsilon > 0$.  Further, construct $v_i^\epsilon: \bbx \times \bby \to [-\frac{1}{\epsilon} , \frac{1}{\epsilon}]$ for any $\epsilon > 0$ by
\begin{equation}\label{Eq:v-eps}
v_i^\epsilon(x,y) := \sup_{y_i^* \in \bby_i} \{f_i^\epsilon(x,y_i^*,y_{-i}) \; | \; y_i^* \in F_i(x,y_{-i})\}
\end{equation}
to be the value that player $i$ receives from this $\epsilon$-bounded game.

The following assumption is stated for ease of reference.  These properties are assumed for the entirety of this section.
\begin{assumption}\label{Ass:Nash-2}
Consider a game following Assumption~\ref{Ass:Nash-1}.  Additionally assume that $\{(x,y_{-i}) \mapsto f_i(x,y_i^*,y_{-i})\}_{y_i^* \in \bby_i}$ is uniformly equicontinuous, i.e., $f_i$ is uniformly continuous in $\bbx \times \bby_{-i}$ uniformly in $\bby_i$.
\end{assumption}

\begin{remark}\label{Rem:Equicontinuous}
If $f_i$ is uniformly continuous then the collection $\{(x,y_{-i}) \mapsto f_i(x,y_i^*,y_{-i})\}_{y_i^* \in \bby_i}$ is uniformly equicontinuous.  In particular, if $\bbx$ and $\bby_j$ are compact spaces for every player $j \in I$ (as is assumed in, e.g.,~\cite{cn2010essential,page2015parameterized} or the motivating example of Section~\ref{Sec:Motivation}) then this holds following from Assumption~\ref{Ass:Nash-1}. 
\end{remark}

\begin{corollary}\label{Cor:veps}
Fix $\epsilon > 0$, then $v_i^\epsilon$ (defined in~\eqref{Eq:v-eps}) is continuous for every player $i \in I$.
\end{corollary}
\begin{proof}
This follows from an application of Theorem~\ref{Thm:MaxCompact}.
\end{proof}

Consider now the \emph{approximate} Nash games defined by the $\vec\epsilon$-best response functions $H_i^{\vec\epsilon}: \bbx \times \bby \to \bby_i$ satisfying
\begin{align}
\label{Eq:H-eps} H_i^{\vec\epsilon}(x,y) &:= \{y_i^* \in \bby_i \; | \; f_i^{\epsilon_2}(x,y_i^*,y_{-i}) > v_i^{\epsilon_2}(x,y) - \epsilon_1, \; y_i^* \in B_{i,\epsilon_3}[F_i(x,y_{-i})]\}
\end{align}
for any $\vec\epsilon \in \bbr^3_{++}$
where $B_{i,\epsilon}[Y_i] := \{y_i \in \bby_i \; | \; \inf_{\bar y_i \in Y_i} d_i(y_i,\bar y_i) < \epsilon\}$ is the (open) ball of size $\epsilon$ surrounding $Y_i \subseteq \bby_i$.
With these approximate best responses, we introduce the associated approximate Nash equilibria
\begin{align}
\label{Eq:h-eps} h^{\vec\epsilon}(x) &:= \FIX_{y \in \bby} H^{\vec\epsilon}(x,y) = \{y \in \bby \; | \; y \in H^{\vec\epsilon}(x,y)\}
\end{align}
for every $x \in \bbx$.  

\begin{remark}
If $\bby_i$ are compact spaces for every $i \in I$ then $\epsilon_2$ is no longer required in \eqref{Eq:H-eps} or \eqref{Eq:h-eps}.  If $F_i \equiv \bby_i$ for every player $i \in I$ then neither $\epsilon_2$ nor $\epsilon_3$ are required for the results of this work.
\end{remark}

With these approximate Nash games, we wish to consider a modification to the original game, i.e.,
\begin{equation}\label{Eq:h-star}
h^*(x) := \lim_{\vec\epsilon \searrow 0} h^{\vec\epsilon}(x) \quad \forall x \in \bbx.
\end{equation}
This set-theoretic limit exists due to the monotonicity of $\vec\epsilon \in \bbr^3_{++} \mapsto h^{\vec\epsilon}(x)$ for a fixed parameter $x \in \bbx$.  In particular, we can characterize this limit by the intersection over approximation errors $\vec\epsilon$, i.e., $h^*(x) := \bigcap_{\vec\epsilon \in \bbr^3_{++}} h^{\vec\epsilon}(x)$ for any $x \in \bbx$.
The below proposition is used to demonstrate that this set of approximating equilibria are, indeed, equal to the original Nash equilibria.
\begin{proposition}\label{Prop:h*}
Consider a parameterized Nash game and its approximations.  For any $x \in \bbx$ the set of Nash equilibria is equivalent to the modified Nash equilibria introduced in~\eqref{Eq:h-star}, i.e., $h = h^*$.
\end{proposition}
\begin{proof}
First, $h(x) \subseteq h^*(x)$ for any $x \in \bbx$ trivially by construction since $h(x) \subseteq h^{\vec\epsilon}(x)$ for every $x \in \bbx$ and every $\vec\epsilon \in \bbr^3_{++}$.
Now consider $y^* \in h^*(x)$. That is, $f_i(x,y^*) > v_i^{\epsilon_2}(x,y^*) - \epsilon_1$ and $y_i^* \in B_{i,\epsilon_3}[F_i(x,y_{-i}^*)]$ for every $i \in I$ and every $\vec\epsilon \in \bbr^3_{++}$.  First, this implies that $y_i^* \in F_i(x,y_{-i}^*)$ (for every player $i \in I$) since $\inf_{\bar y_i \in F_i(x,y_{-i}^*)} d_i(y_i^*,\bar y_i) < \epsilon_3$ for every $\epsilon_3 > 0$.  Therefore to prove that $y^* \in h(x)$, it is sufficient to show that 
\[f_i(x,y^*) = v_i(x,y^*) := \sup_{y_i \in \bby_i} \left\{f_i(x,y_i,y_{-i}^*) \; | \; y_i \in F_i(x,y_{-i}^*)\right\}\] for every player $i \in I$.  
To complete this proof we consider 2 cases:
\begin{enumerate}
\item First, assume $v_i(x,y^*) \in \bbr$ for every player $i \in I$.  Let $\epsilon_2 < \frac{1}{\max_{i \in I} |v_i(x,y^*)|}$ (with $1/0 = +\infty$), then $f_i^{\epsilon_2}(x,y^*) = f_i(x,y^*)$ and $v_i^{\epsilon_2}(x,y^*) = v_i(x,y^*)$ for every player $i \in I$ since $f_i(x,y^*) \leq v_i(x,y^*)$ by definition.  Thus, by construction of $y^* \in h^*(x)$, it follows that $f_i(x,y^*) > v_i(x,y^*) - \epsilon_1$ for every $\epsilon_1 > 0$ and every player $i \in I$.  As such, $f_i(x,y^*) \geq v_i(x,y^*)$ and, as previously mentioned, $f_i(x,y^*) \leq v_i(x,y^*)$ by definition.
\item Second, assume there exists some player $j \in I$ such that $v_j(x,y^*) = +\infty$.  By construction, this implies implies $h(x) = \emptyset$, thus we wish to show a contradiction to $y^* \in h^*(x)$.  Fix some $j \in I$ such that $v_j(x,y^*) = +\infty$.  Consider $\epsilon_2 < \frac{1}{|f_j(x,y^*)|}$ so that $f_j^{\epsilon_2}(x,y^*) = f_j(x,y^*)$ and $v_j^{\epsilon_2}(x,y^*) = \frac{1}{\epsilon_2} > f_j(x,y^*)$.  Consider also $\epsilon_1 < \frac{1}{\epsilon_2} - f_j(x,y^*)$.  Then, by construction of $\vec\epsilon$, $y_j^* \not\in H_j^{\vec\epsilon}(x,y^*)$ contradicting the assumption that $y^* \in h^{\vec\epsilon}(x)$ providing the contradiction.  
\end{enumerate}
\end{proof}

This brings us, now, to the main result of this work.  Namely, that we can characterize the set of all Nash equilibria under approximations of the game and approximations of the parameter. For ease of reference, the assumptions utilized in the below theorem are summarized in Assumptions~\ref{Ass:Nash-1} and~\ref{Ass:Nash-2}.
\begin{theorem}\label{Thm:Approx-Limit}
Let $(x_n)_{n \in \bbn} \to x$ be a convergent sequence of parameters in the domain of $h$, i.e., so that $h(x_n),h(x) \neq \emptyset$ for every $n \in \bbn$.  Then
\begin{align*}
h(x) &= \lim_{\vec\epsilon \searrow 0} \liminf_{n \to \infty} h^{\vec\epsilon}(x_n) = \lim_{\vec\epsilon \searrow 0} \limsup_{n \to \infty} h^{\vec\epsilon}(x_n).
\end{align*}
\end{theorem}
\begin{proof}
First, we wish to note that, for every player $i \in I$, $(x,y^*) \mapsto \inf_{\bar y_i \in F_i(x,y_{-i}^*)} [d_i(y_i^*,\bar y_i) \wedge (2 \epsilon_3)]$ is continuous by an application of Theorem~\ref{Thm:MaxCompact} since the modified metric objective is appropriately equicontinuous as a collection over $\bar y$. 

For fixed approximation level $\vec\epsilon \in \bbr^3_{++}$, $H^{\vec\epsilon}$ and $h^{\vec\epsilon}$ have open graphs in their respective topologies.
Consider the graph of $H^{\vec\epsilon}$:
\begin{align*}
\graph H^{\vec\epsilon} &= \left\{(x,y,y^*) \in \bbx \times \bby \times \bby \; \left| \; \begin{array}{ll} \forall i \in I: & f_i^{\epsilon_2}(x,y_i^*,y_{-i}) > v_i^{\epsilon_2}(x,y) - \epsilon_1, \\ & y_i^* \in B_{i,\epsilon_3}[F_i(x,y_{-i})]\end{array}\right.\right\}\\
&= \bigcap_{i \in I} \{(x,y,y^*) \in \bbx \times \bby \times \bby \; | \; f_i^{\epsilon_2}(x,y_i^*,y_{-i}) > v_i^{\epsilon_2}(x,y) - \epsilon_1\} \cap\\
	&\qquad\qquad \bigcap_{i \in I} \{(x,y,y^*) \in \bbx \times \bby \times \bby \; | \; \inf_{\bar y_i \in F_i(x,y_{-i})} d_i(y_i^*,\bar y_i) < \epsilon_3\}\\
&= \bigcap_{i \in I} \{(x,y,y^*) \in \bbx \times \bby \times \bby \; | \;\epsilon_1 > v_i^{\epsilon_2}(x,y) - f_i^{\epsilon_2}(x,y_i^*,y_{-i})\} \cap\\
	&\qquad\qquad \bigcap_{i \in I} \{(x,y,y^*) \in \bbx \times \bby \times \bby \; | \; \inf_{\bar y_i \in F_i(x,y_{-i})} [d_i(y_i^*,\bar y_i) \wedge (2 \epsilon_3)] < \epsilon_3\}
\end{align*}
is the finite intersection of open sets by continuity of all involved functions.
Finally, $h^{\vec\epsilon}$ inherits this property from $H^{\vec\epsilon}$ as provided in Lemma~\ref{Lemma:FixedPtLC}\eqref{Lemma:FixedPtLC-1}.

In order to complete this proof, let us introduce a second (closed) $\vec\epsilon$-approximate best response function $\bar H_i^{\vec\epsilon}: \bbx \times \bby \to \bby_i$ satisfying
\begin{align*}
\bar H_i^{\vec\epsilon}(x,y) &:= \{y_i^* \in \bby_i \; | \; f_i^{\epsilon_2}(x,y_i^*,y_{-i}) \geq v_i^{\epsilon_2}(x,y) - \epsilon_1, \; y_i^* \in \cl B_{i,\epsilon_3}[F_i(x,y_{-i})]\}
\end{align*}
for any $\vec\epsilon \in \bbr^3_{++}$.
As before, with these approximate best responses, we introduce the associated approximate Nash equilibria
\begin{align}
\label{Eq:h-eps-cl} \bar h^{\vec\epsilon}(x) &:= \{y^* \in \bby \; | \; y^* \in \bar H^{\vec\epsilon}(x,y^*)\}
\end{align}
for every $x \in \bbx$.  
Note that, by construction, $h(x) \subseteq h^{\vec\epsilon}(x) \subseteq \bar h^{\vec\epsilon}(x) \subseteq h^{2 \times \vec\epsilon}(x)$ for any $x \in \bbx$ and any approximation error $\vec\epsilon \in \bbr^3_{++}$.

Now, for fixed approximation level $\vec\epsilon \in \bbr^3_{++}$, $\bar H^{\vec\epsilon}$ and $\bar h^{\vec\epsilon}$ have closed graphs in their respective topologies.
Consider the graph of $\bar H^{\vec\epsilon}$:
\begin{align*}
\graph \bar H^{\vec\epsilon} &= \left\{(x,y,y^*) \in \bbx \times \bby \times \bby \; \left| \; \begin{array}{ll} \forall i \in I: & f_i^{\epsilon_2}(x,y_i^*,y_{-i}) \geq v_i^{\epsilon_2}(x,y) - \epsilon_1, \\ & y_i^* \in \cl B_{i,\epsilon_3}[F_i(x,y_{-i})]\end{array}\right.\right\}\\
&= \bigcap_{i \in I} \{(x,y,y^*) \in \bbx \times \bby \times \bby \; | \; f_i^{\epsilon_2}(x,y_i^*,y_{-i}) \geq v_i^{\epsilon_2}(x,y) - \epsilon_1\} \cap\\
	&\qquad\qquad \bigcap_{i \in I} \{(x,y,y^*) \in \bbx \times \bby \times \bby \; | \; \inf_{\bar y_i \in F_i(x,y_{-i})} d_i(y_i^*,\bar y_i) \leq \epsilon_3\}\\
&= \bigcap_{i \in I} \{(x,y,y^*) \in \bbx \times \bby \times \bby \; | \;\epsilon_1 \geq v_i^{\epsilon_2}(x,y) - f_i^{\epsilon_2}(x,y_i^*,y_{-i})\} \cap\\
	&\qquad\qquad \bigcap_{i \in I} \{(x,y,y^*) \in \bbx \times \bby \times \bby \; | \; \inf_{\bar y_i \in F_i(x,y_{-i})} [d_i(y_i^*,\bar y_i) \wedge (2 \epsilon_3)] \leq \epsilon_3\}
\end{align*}
is the intersection of closed sets by continuity of all involved functions.
Finally, $\bar h^{\vec\epsilon}$ inherits this property from $\bar H^{\vec\epsilon}$ as provided in Lemma~\ref{Lemma:FixedPtUC}\eqref{Lemma:FixedPtUC-1}.

Therefore by construction of the set-theoretic limit and Propositions~\ref{Prop:h*},~\ref{Prop:ClosedLimit}, and~\ref{Prop:LCLimit},
\begin{align*}
h(x) &= \bigcap_{\vec\epsilon \in \bbr^3_{++}} h^{\vec\epsilon}(x) \subseteq \bigcap_{\vec\epsilon \in \bbr^3_{++}} \liminf_{n \to \infty} h^{\vec\epsilon}(x_n) \subseteq \bigcap_{\vec\epsilon \in \bbr^3_{++}} \limsup_{n \to \infty} h^{\vec\epsilon}(x_n) \\ 
&\subseteq \bigcap_{\vec\epsilon \in \bbr^3_{++}} \limsup_{n \to \infty} \bar h^{\vec\epsilon}(x_n) \subseteq \bigcap_{\vec\epsilon \in \bbr^3_{++}} \bar h^{\vec\epsilon}(x) \subseteq \bigcap_{\vec\epsilon \in \bbr^3_{++}} h^{2 \times \vec\epsilon}(x) = \bigcap_{\vec\epsilon \in \bbr^3_{++}} h^{\vec\epsilon}(x) = h(x).
\end{align*}
\end{proof}

\begin{remark}\label{Rem:Appox-Limit-Closed}
A small modification of the proof of Theorem~\ref{Thm:Approx-Limit} provides the alternative limiting result with respect to the \emph{closed} approximate Nash equilibria $\bar h^{\vec\epsilon}$ (defined in~\eqref{Eq:h-eps-cl}), i.e.,
\[h(x) = \lim_{\vec\epsilon \searrow 0} \liminf_{n \to \infty} \bar h^{\vec\epsilon}(x_n) = \lim_{\vec\epsilon \searrow 0} \limsup_{n \to \infty} \bar h^{\vec\epsilon}(x_n)\]
for $(x_n)_{n \in \bbn} \to x$ in the domain of $h$.
\end{remark}

\section{Conclusion}\label{Sec:Conclusion}
In this work we have constructed order relations for the set of all Nash equilibria.  In particular, we found that, for very general set of games, the limit of parameterized Nash equilibria are only a subset of all true equilibria. We then find certain sufficient conditions so that the true equilibria can be found as the limit of approximate parameterized Nash equilibria.  As such, we argue that more focus should be on the approximate Nash equilibria proposed in this work due to this convergence analysis and the inherent nature of the estimation of game parameters.  Notably, none of the results in this work rely on uniqueness or the compactness of the strategy spaces.  As such, these results are applicable for, e.g., mean field games.

As introduced in Example~\ref{Ex:Nutz} and discussed more thoroughly in \cite{nutz2018casestudy}, the mean field limit can result in more equilibria than those that are the limit of $n$ player games (as $n \to \infty$).  However, given a mean field game, Theorem~\ref{Thm:Approx-Limit} is not strong enough to determine a robust approximate $n$ player game for the desired mean field game.  We leave this as an important extension to consider in future works.

\section*{Acknowledgements}
The author would like to thank Jianfeng Zhang for many discussions on this subject.

\appendix
\section{Set-Valued Continuity}\label{Sec:Background}
In this section we will present definitions of continuity for set-valued or multivalued mappings, i.e., functions mapping into the power set of some space.  We will additionally provide a brief overview of some results on these forms of continuity from the literature.  Throughout this section we will let $\bbx$ and $\bby$ be Hausdorff spaces.  

\begin{definition}\label{Defn:Continuity}
A set-valued mapping $F: \bbx \to \pcal(\bby)$ is called \emph{(upper, lower) continuous} if it is continuous with respect to the (resp.\ upper, lower) Vietoris topology.
\end{definition}

\begin{remark}\label{Rem:ContDefn}
The mapping $F: \bbx \to \pcal(\bby)$ is continuous if, and only if, it is upper and lower continuous.  If $F$ is single-valued, i.e., $F(x) = \{f(x)\}$ for some function $f: \bbx \to \bby$, then $F$ is upper continuous if, and only if, $F$ is lower continuous if, and only if, $f$ is continuous.
\end{remark}

\begin{remark}
Upper (lower) continuity is often referred to as upper (resp.\ lower) hemicontinuity (in, e.g., \cite{AB07}), upper (resp.\ lower) semicontinuity (in, e.g., \cite{AF90,HP97,JR02selectors,RS98selectors}), and inner (resp.\ outer) continuity (in, e.g., \cite{RW09}) in the literature.  We use the terminology from \cite{KTZ15,GRTZ03,HS12-uc} to emphasize that a single-valued function is upper continuous if, and only if, it is lower continuous if, and only if, it is continuous.  This has the additional advantage of avoiding the need to distinguish single-valued semicontinuity from set-valued continuity concepts.
\end{remark}

The following equivalent representations for upper and lower continuity are standard in the literature (see \cite[Propositions 1.2.6 and 1.2.7]{HP97} and \cite[Lemmas 17.4 and 17.5]{AB07}).
\begin{proposition}\label{Prop:UC}
For a set-valued mapping $F: \bbx \to \pcal(\bby)$ the following are equivalent:
\begin{enumerate}
\item $F$ is upper continuous;
\item $F^+[V] := \{x \in \bbx \; | \; F(x) \subseteq V\}$ is open in $\bbx$ for any $V \subseteq \bby$ open;
\item $F^-[\bar V] := \{x \in \bbx \; | \; F(x) \cap \bar V \neq \emptyset\}$ is closed in $\bbx$ for any $\bar V \subseteq \bby$ closed.
\end{enumerate}
\end{proposition}
\begin{proposition}\label{Prop:LC}
For a set-valued mapping $F: \bbx \to \pcal(\bby)$ the following are equivalent:
\begin{enumerate}
\item $F$ is lower continuous;
\item $F^-[V] := \{x \in \bbx \; | \; F(x) \cap V \neq \emptyset\}$ is open in $\bbx$ for any $V \subseteq \bby$ open;
\item $F^+[\bar V] := \{x \in \bbx \; | \; F(x) \subseteq \bar V\}$ is closed in $\bbx$ for any $\bar V \subseteq \bby$ closed.
\end{enumerate}
\end{proposition}

We additionally wish to provide simple conditions for upper and lower continuity via the graphs of our multivalued functions.
\begin{theorem}\label{Thm:ClosedGraphThm}
Let $\bby$ be a regular topological space.  Consider the set-valued mapping $F: \bbx \to \hat\pcal_f(\bby) := \{Y \in \pcal(\bby) \; | \; Y = \cl(Y)\}$.
\begin{enumerate}
\item $F$ has a closed graph, i.e.,
	\[\graph F := \{(x,y) \in \bbx \times \bby \; | \; y \in F(x)\} \text{ is closed in } \bbx \times \bby,\]
	if $F$ is upper continuous;
\item $F$ is upper continuous if $\bby$ is a compact space and $F$ has a closed graph.
\end{enumerate}
\end{theorem}
\begin{proof}
This follows from \cite[Proposition 2.17]{HP97} and \cite[Theorem 17.11]{AB07}.
\end{proof}

\begin{lemma}\cite[Lemma 17.12]{AB07}\label{Lemma:OpenGraphLemma}
For set-valued mappings $F: \bbx \to \pcal(\bby)$ we have that $F$ is lower continuous if $F$ has open fibers, i.e.,
\[F^{-}(\{y\}) := \{x \in \bbx \; | \; y \in F(x)\}\]
is open for any $y \in \bby$.  And $F$ has open fibers if the graph of $F$ is open in the product topology.
\end{lemma}

We conclude this section by providing comparison of continuity results with Kuratowski limits.
\begin{proposition}\cite[Proposition 2.13]{HP97}\label{Prop:ClosedLimit}
A set-valued mapping $F: \bbx \to \pcal(\bby)$ has a closed graph if, and only if,
\[F(x) \supseteq \limsup_{i \in I} F(x_i) := \bigcap_{i \in I} \cl \bigcup_{j \geq i} F(x_j)\]
for every net $(x_i)_{i \in I} \to x$ in the domain of $F$.
\end{proposition}

\begin{proposition}\label{Prop:LCLimit}
If a set-valued mapping $F: \bbx \to \pcal(\bby)$ is lower continuous then
\[F(x) \subseteq \liminf_{i \in I} F(x_i) := \{y \in \bby \; | \; \forall V \in \ncal(y) \; \exists i \in I: \; F(x_j) \cap V \neq \emptyset \; \forall j \geq i\}\]
for every net $(x_i)_{i \in I} \to x$ in the domain of $F$ where $\ncal(y)$ is the set of neighborhoods of $y \in \bby$.
\end{proposition}
\begin{proof}
Let $y \in F(x)$ and take some neighborhood $V \in \ncal(y)$.  By lower continuity it follows that $F^-[V]$ is a neighborhood of $x$.  Therefore, there exists some $i \in I$ such that $x_j \in F^-[V]$ for every $j \geq i$.
\end{proof}

\section{Continuity of the value function and best response function}\label{Sec:BRF}
In this section we will present continuity results for parameterized optimization problems.  From the literature we can derive results on the value function (and in the case of the Berge maximum theorem, the optimizers) of such an optimization problem.  Throughout this section we will let $\bbx$ and $\bby$ be Hausdorff spaces except where otherwise indicated.
\begin{theorem}\label{Thm:MaxTheorem}
Let $f: \bbx \times \bby \to \bbr$ be some desired objective function and $F: \bbx \to \pcal(\bby)$.  Let $v: \bbx \to \bbr \cup \{\pm\infty\}$ be defined by
\begin{equation}\label{Eq:MaxTheorem}
v(x) := \sup\{f(x,y) \; | \; y \in F(x)\}
\end{equation}
for any $x \in \bbx$.
\begin{enumerate}
\item \label{Thm:MaxTheorem-1} The value function $v$ is upper semicontinuous if $f$ is upper semicontinuous and $F$ is upper continuous with nonempty and compact images.
\item \label{Thm:MaxTheorem-2} The value function $v$ is lower semicontinuous if $f$ is lower semicontinuous and $F$ is lower continuous.
\item \label{Thm:MaxTheorem-3} The value function $v$ is continuous if $f$ is continuous and $F$ is continuous with nonempty and compact images.  Further, the set of maximizers $V: \bbx \to \pcal(\bby)\backslash\{\emptyset\}$, defined by
\begin{equation}\label{Eq:MaxTheorem-2}
V(x) := \argmax\{f(x,y) \; | \; y \in F(x)\} = \{y \in F(x) \; | \; f(x,y) = v(x)\}
\end{equation}
for every $x \in \bbx$, is upper continuous with compact images.
\end{enumerate}
\end{theorem}
\begin{proof}
These are trivial consequences of \cite[Lemmas 17.29 and 17.30]{AB07} and the Berge maximum theorem (see, e.g., \cite[Theorem 17.31]{AB07}).
\end{proof}

The following theorem provides an extension of Theorem~\ref{Thm:MaxTheorem}\eqref{Thm:MaxTheorem-3} in that it does \emph{not} require the compactness of the space $\bby$.  Under compactness of the space $\bby$, these results are equivalent by application of the closed graph theorem (Theorem~\ref{Thm:ClosedGraphThm}). 
\begin{theorem}\label{Thm:MaxGraph}
Let $f: \bbx \times \bby \to \bbr$ be continuous and $F: \bbx \to \pcal_f(\bby)$ be continuous.  The set of maximizers $V: \bbx \to \pcal(\bby)$, defined by \eqref{Eq:MaxTheorem-2}, has a closed graph in the product topology $\bbx \times \bby$.
\end{theorem}
\begin{proof}
First, recall that $\graph V = \{(x,y) \in \bbx \times \bby \; | \; y \in V(x)\}$.  Consider the convergent net $(x_i,y_i)_{i \in I} \subseteq \graph V \to (x,y) \in \bbx \times \bby$.  To prove $\graph V$ is closed, we wish to show that $(x,y) \in \graph V$.  That is, $y \in F(x)$ and $f(x,y) \geq f(x,\hat y)$ for every $\hat y \in F(x)$.
\begin{enumerate}
\item By $F$ upper continuous with nonempty and closed values, the closed graph theorem (Theorem~\ref{Thm:ClosedGraphThm}) implies $\graph F$ is closed.  This implies that $(x,y) \in \graph F$ by $(x_i,y_i) \in \graph F$ for every $i \in I$, i.e., $y \in F(x)$.
\item Suppose $y \not\in V(x)$.  Then there exists some $\hat y \in F(x)$ such that $f(x,y) < f(x,\hat y)$.  By the lower continuity of $F$, there exists a subnet $(x_{i_j})_{j \in J} \to x$ and net $(\hat y_j)_{j \in J} \to \hat y$ such that $\hat y_j \in F(x_{i_j})$ for every $j \in J$.  By the continuity of the objective function we find:
\begin{align*}
\lim_{j \in J} f(x_{i_j},\hat y_j) &= f(x,\hat y) > f(x,y) = \lim_{j \in J} f(x_{i_j},y_{i_j}).
\end{align*}
Therefore for sufficiently large $j \in J$ it must follow that $f(x_{i_j},y_{i_j}) < f(x_{i_j},\hat y_j)$ which contradicts the initial assumption that $(x_{i_j},y_{i_j}) \in \graph V$ thus completing the proof.
\end{enumerate}
\end{proof}

\begin{theorem}\label{Thm:MaxCompact}
Let $\bbx$ be a metric space and let $\bby$ be a Tychonoff space.  Let $f: \bbx \times \bby \to K \subseteq \bbr$ for compact set $K$ be continuous and such that $\{f(\cdot,y)\}_{y \in \bby}$ is uniformly equicontinuous.
Additionally, let $F: \bbx \to \pcal(\bby)\backslash\{\emptyset\}$ be continuous.  The value function $v: \bbx \to \bbr \cup \{\pm \infty\}$, defined by \eqref{Eq:MaxTheorem}, is continuous.
\end{theorem}
\begin{proof}
Define $\bar\bby$ to be the Stone-\v{C}ech compactification of $\bby$ with embedding $\beta: \bby \to \bar\bby$.  Let $\bar f_x: \bar\bby \to K$ be the unique continuous extension of $f(x,\cdot)$ for fixed $x \in \bbx$ (see, e.g., \cite[Theorem 2.79]{AB07}).
Define $\bar f: \bbx \times \bar\bby \to K$ by $\bar f(x,\bar y) := \bar f_x(\bar y)$ for any $x \in \bbx$ and $\bar y \in \bar\bby$.  By uniform equicontinuity and $\beta(\bby)$ being a dense subset of $\bar\bby$, we will show that $\bar f$ is a continuous mapping.  Take a net $(x_i,\bar y_i)_{i \in I} \to (x,\bar y)$ in $\bbx \times \bar\bby$ and net $(y_i^j)_{j \in J}$ such that $\lim_{j \in J} \beta(y_i^j) = \bar y_i$ for every $i \in I$, then
\begin{align*}
\lim_{i \in I} |\bar f(x,\bar y) - \bar f(x_i,\bar y_i)| &= \lim_{i \in I} |\bar f_x(\bar y) - \bar f_{x_i}(\bar y_i)|\\
&\leq \lim_{i \in I} \left(|\bar f_x(\bar y) - \bar f_x(\bar y_i)| + |\bar f_x(\bar y_i) - \bar f_{x_i}(\bar y_i)|\right)\\
&= \lim_{i \in I} |\bar f_x(\bar y_i) - \bar f_{x_i}(\bar y_i)|\\
&= \lim_{i \in I} \lim_{j \in J} |f(x,y_i^j) - f(x_i,y_i^j)|\\
&= \lim_{j \in J} \lim_{i \in I} |f(x,y_i^j) - f(x_i,y_i^j)| = 0.
\end{align*}
Define $\bar F: \bbx \to \pcal(\bar\bby)$ by $\bar F(x) := \cl\{\beta(y) \; | \; y \in F(x)\}$ with closure taken in $\bar\bby$.  Trivially, by construction, $\bar F$ has nonempty and closed (and therefore compact) images.  We will now show that $\bar F$ is continuous.  Define $G: \bbx \to \pcal(\bar\bby)$ by $G(x) := \beta[F(x)]$.  By construction, $G^-[\cdot] = F^-[\beta^{-1}[\cdot]]$ where $\beta^{-1}(\bar y) = \emptyset$ if $\bar y \in \bar\bby \backslash \beta[\bby]$.  By continuity of $F$ and $\beta$, it is trivial to see that $G$ must also be continuous.  Because $\bar F(x) = \cl G(x)$ it is continuous since the closure of a lower continuous function is lower continuous (\cite[Proposition 2.38]{HP97}) and the closure of an upper continuous function mapping into a normal range space (true of $\bar\bby$ as it is a compact Hausdorff space) is upper continuous (\cite[Proposition 2.40]{HP97}).
Finally, we wish to utilize Theorem~\ref{Thm:MaxTheorem}\eqref{Thm:MaxTheorem-3} in order to prove the result.  Using this result we note that $\bar v: \bbx \to K$, defined by
\[\bar v(x) := \sup\{\bar f(x,\bar y) \; | \; \bar y \in \bar F(x)\},\]
is continuous.  And by property of the supremum, and that $\bar f$ is an extension of $f$, we conclude
\begin{align*}
\bar v(x) &= \sup\{\bar f(x,\bar y) \; | \; \bar y \in \bar F(x)\}\\
&= \sup\{\bar f(x,\bar y) \; | \; \bar y \in G(x)\}\\
&= \sup\{\bar f(x,\beta(y)) \; | \; y \in F(x)\}\\
&= \sup\{f(x,y) \; | \; y \in F(x)\} = v(x).
\end{align*}
\end{proof}

\section{Continuity of parameterized fixed points}\label{Sec:FixedPt}
In this section we will provide continuity results and simple sensitivity analysis for the collection of fixed points for set-valued mappings.  These results follow from the definitions of continuity given in the prior section for the set-valued mapping.  
Throughout this section we will let $\bbx$ and $\bby$ be Hausdorff spaces except where otherwise indicated.
Additionally throughout, consider the fixed points of the mapping $H: \bbx \times \bby \to \pcal(\bby)$.  We will denote the parameterized fixed points of $H$ by the function $h: \bbx \to \pcal(\bby)$, i.e., 
\[h(x) := \FIX_{y \in \bby} H(x,y) = \{y \in \bby \; | \; y \in H(x,y)\} \quad \forall x \in \bbx.\]

The results in this section are generally trivial and highly related to the literature on data dependence of fixed points in, e.g., \cite{markin73,kirr97,rus14,PRS14}.  Though not the main focus of this work, the continuity results presented in this section are widely applicable in, e.g., convergence of PageRank~\cite{GHL2018pagerank} with respect to the underlying network topology or data dependence of financial systemic risk models such as those in \cite{EN01,CFS05,AW15,F15illiquid} where methodology to estimate system parameters are studied in, e.g., \cite{GV16}.  

\begin{lemma}\label{Lemma:FixedPtUC}
\begin{enumerate}
\item \label{Lemma:FixedPtUC-1} If $\graph H \subseteq \bbx \times \bby \times \bby$ is closed in the product topology then $\graph h \subseteq \bbx \times \bby$ is closed in the product topology.
\item \label{Lemma:FixedPtUC-2} If the properties of \eqref{Lemma:FixedPtUC-1} are satisfied and $\bby$ is a compact Hausdorff space then $h$ is an upper continuous multivalued map with closed and compact images.
\end{enumerate}
\end{lemma}
\begin{proof}
\begin{enumerate}
\item Recall that the graph of $h$ is given by
\[\graph h := \{(x,y) \in \bbx \times \bby \; | \; y \in h(x)\}.\]
Let $\{(x_i,y_i)\}_{i \in I} \subseteq \bbx \times \bby \to (x,y)$ such that $(x_i,y_i) \in \graph h$ for every $i \in I$.  By definition of the mapping $h$ it is immediate that $(x_i,y_i,y_i) \in \graph H$ for every $i \in I$.  By convergence in the product topology and closedness of the graph of $H$ it immediately follows that $y \in H(x,y)$, i.e., $y \in h(x)$.
\item If we additionally assume that $\bby$ is compact then we can apply the closed graph theorem (Theorem~\ref{Thm:ClosedGraphThm}) to recover that $h$ is upper continuous and closed-valued.  Since a closed subset of a compact set is compact, we recover that $h$ is additionally compact-valued.
\end{enumerate}
\end{proof}

\begin{remark}\label{Rem:Nonempty}
If the properties of Lemma~\ref{Lemma:FixedPtUC}\eqref{Lemma:FixedPtUC-2} are satisfied, $\bby$ is a locally convex space that is convex, and $H$ has nonempty convex images then $h$ has nonempty images by the Kakutani fixed point theorem (see, e.g., \cite[Corollary 17.55]{AB07})
\end{remark}

\begin{lemma}\label{Lemma:FixedPtLC}
\begin{enumerate}
\item \label{Lemma:FixedPtLC-1} If $\graph H \subseteq \bbx \times \bby \times \bby$ is open in the product topology then $\graph h \subseteq \bbx \times \bby$ is open in the product topology.
\item \label{Lemma:FixedPtLC-2} If $H$ has open fibers (i.e., $H^{-}(\{\bar y\}) := \{(x,y) \in \bbx \times \bby \; | \; \bar y \in H(x,y)\}$ is open for every $\bar y \in \bby$) then $h$ has open fibers.
\end{enumerate}
\end{lemma}
\begin{proof}
\begin{enumerate}
\item Let $\{(x_i,y_i)\}_{i \in I} \subseteq \bbx \times \bby \to (x,y)$ such that $(x_i,y_i) \not\in \graph h$ for every $i \in I$.  By definition of the mapping $h$ it is immediate that $(x_i,y_i,y_i) \not\in \graph H$ for every $i \in I$.  By convergence in the product topology and openness of the graph of $H$ it immediately follows that $y \not\in H(x,y)$, i.e., $y \not\in h(x)$.
\item Fix $\bar y \in \bby$.  Let $\{x_i\}_{i \in I} \subseteq \bbx \to x$ such that $x_i \not\in h^{-}(\{\bar y\})$ for every $i \in I$.  By definition of the mapping $h$ it is immediate that $(x_i,\bar y) \not\in H^{-}(\{\bar y\})$ for every $i \in I$.  By convergence in the product topology and openness of the fibers of $H$ it immediately follows that $(x,\bar y) \not\in H^{-}(\{\bar y\})$, i.e., $\bar y \not\in h(x)$.
\end{enumerate}
\end{proof}

\begin{remark}
The condition of Lemma~\ref{Lemma:FixedPtLC}\eqref{Lemma:FixedPtLC-1} and \eqref{Lemma:FixedPtLC-2} imply $h$ is lower continuous, in fact the condition of Lemma~\ref{Lemma:FixedPtLC}\eqref{Lemma:FixedPtLC-1} implies \eqref{Lemma:FixedPtLC-2} (see, e.g., Lemma~\ref{Lemma:OpenGraphLemma}).
\end{remark}

\bibliographystyle{plain}
\bibliography{biblio}

\end{document}